\documentclass[12pt]{article}

\usepackage{graphicx}
\usepackage{amsmath,amsthm,amssymb}
\usepackage{mathrsfs}

\usepackage{hyperref}

\usepackage{subfigure}

\usepackage[left=2cm,right=2cm,top=3.5cm,bottom=3.5cm]{geometry}

\usepackage{color}
\usepackage{array}

\theoremstyle{plain}
\newtheorem{thm}{Theorem}[section]
\newtheorem{prop}[thm]{Proposition}

\newtheorem{cor}[thm]{Corollary}
\theoremstyle{definition}
\newtheorem{defn}[thm]{Definition}

\newtheorem*{ex*}{Example}
\theoremstyle{remark}

\usepackage{csquotes}

\numberwithin{equation}{section}

\newcommand{\dd}{\,{\rm d}}

\newcommand{\id}{{\rm Id}}
\newcommand{\tr}{{\rm tr}}
\renewcommand{\S}{\mathcal{S}^{2\times 2}}
\newcommand{\Sz}{\mathcal{S}^{2\times 2}_0}
\renewcommand{\div}{{\rm div}}
\newcommand{\half}{\frac{1}{2}}

\begin{document}

\title{The 2-d isentropic compressible Euler equations may have infinitely many solutions which conserve energy}

\author{Simon Markfelder
\and Christian Klingenberg 
}

\date{\today}

\maketitle

\bigskip

\centerline{Dept. of Mathematics, W\"urzburg University, Germany}

\bigskip

\begin{abstract} 
	We consider the 2D isentropic compressible Euler equations. It was shown in \cite{chio15} that there exist Riemann initial data as well as Lipschitz initial data for which there exist infinitely many weak solutions that fulfill an energy \emph{in}equality. In this note we will prove that there is Riemann initial data for which there exist infinitely many weak solutions that conserve energy, i.e. they fulfill an energy equality. As in the aforementioned paper we will also show that there even exists Lipschitz initial data with the same property.
\end{abstract}


\tableofcontents 

\section{Introduction and main result} 
For hyperbolic conservation laws the quest for showing that the initial value problem is well-posed is more than a century old. The focus has been on the compressible Euler equations in particular. In one space dimensions the seminal work of Glimm \cite{glimm65} and DiPerna \cite{diperna83} gave rise to the hope that this goal might be achievable. In the two dimensional case progress has been made to this end for particular self similar solutions, see Chen and Feldman \cite{chen16}. In all these cases an admissibility condition, which mimics the 2nd law of thermodynamics, needs to be invoked. Thus results for the equations of two-dimensional isentropic compressible gas dynamics, that prove that with these admissibility conditions the initial value problem may lead to infinitely many admissible weak solutions (see \cite{dls10}) have been met with quite a surprise. Even in situations, where one can solve the initial value problem explicitly via an admissible solution, one could prove that in addition to this standard solution infinitely many other admissible weak solutions exist \cite{chio15}, \cite{chio14} and \cite{mk17}. The hunt is open for finding new admissibility conditions proving this line of research a fluke. 

The admissibility criterion used so far is entropy dissipation. For the isentropic Euler equations energy takes the role of entropy, so this means that one was seeking weak solutions that dissipate energy. So could one rid oneself of these infinitely many extra solution by demanding the more physical criterion that the solutions conserve energy? In this paper we show that this is not the case. We can show that there exist Lipschitz initial data that may lead to infinitely many solutions of the two dimensional Euler equations all of which conserve energy.

To this end we consider the Cauchy problem for the 2-d isentropic compressible Euler system, i.e. 
\begin{equation}
	\begin{split}
		\partial_t \varrho + \div_x (\varrho\,v) &= 0, \\
		\partial_t (\varrho\,v) + \div_x(\varrho\,v\otimes v)+\nabla_x p(\varrho) &= 0, \\
		\varrho(0,x) &= \varrho_0 (x), \\
		v(0,x) &= v_0(x)
	\end{split}
	\label{eq:euler}
\end{equation}
where the unknowns, more precisely the density $\varrho=\varrho(t,x)\in\mathbb{R}^+$ and the velocity $v=v(t,x)\in\mathbb{R}^2$, are functions of time $t\in[0,\infty)$ and position $x=(x_1,x_2)\in\mathbb{R}^2$. We will focus both on Riemann and on Lipschitz initial data $(\varrho_0,v_0)$, see below. 

Furthermore we consider the polytropic pressure law, more specifically $p(\varrho)=\varrho^2$. However the results can be extended to more general pressure laws 
\begin{equation}
	p(\varrho)=K\,\varrho^\gamma,
\label{eq:pressure}
\end{equation}
where $K>0$ and $\gamma\geq 1$ are constants. 

In this paper the aim is to find energy conserving weak solutions, which are defined as follows.

\begin{defn} \emph{(energy conserving weak solution)}
	A weak solution to \eqref{eq:euler} is energy conserving if the energy equality 
	\begin{equation}
		\partial_t\bigg(\half\,\varrho\,|v|^2 + P(\varrho)\bigg)+\div_x\bigg[\bigg(\half\,\varrho\,|v|^2 + P(\varrho)+p(\varrho)\bigg)v\bigg] = 0,
	\label{eq:eneq}
	\end{equation}
	holds in the weak sense, where the pressure potential is given by $P(\varrho)=\varrho^2$. In the case of more general pressures \eqref{eq:pressure} the pressure potential reads $P(\varrho)=\frac{K}{\gamma-1}\,\varrho^\gamma$ in the case $\gamma>1$ and $P(\varrho)=K\,\varrho\,\log(\varrho)$ for $\gamma=1$.
\end{defn}

In the first part of the paper we consider Riemann initial data, i.e. 
\begin{equation}
\begin{split}
	(\varrho,v)(0,x) &= (\varrho_0,v_0)(x):=\left\{
	\begin{array}[c]{ll}
	(\varrho_-,v_-) & \text{ if }x_2<0 \\
	(\varrho_+,v_+) & \text{ if }x_2>0
	\end{array}
	\right. ,
\end{split}
\label{eq:riemann}
\end{equation}
where $\varrho_\pm\in\mathbb{R}^+$ and $v_\pm\in\mathbb{R}^2$ are constants.

In the past the system \eqref{eq:euler} with initial data \eqref{eq:riemann} was often discussed concerning uniqueness of so-called \emph{admissible} weak solutions defined as follows.

\begin{defn} \emph{(admissible weak solution)} 
	A weak solution to \eqref{eq:euler} is admissible if the energy \emph{in}equality
	\begin{equation}
		\partial_t\bigg(\half\,\varrho\,|v|^2 + P(\varrho)\bigg)+\div_x\bigg[\bigg(\half\,\varrho\,|v|^2 + P(\varrho)+p(\varrho)\bigg)v\bigg] \leq 0.
	\label{eq:enineq}
	\end{equation}
	holds in the weak sense.
\end{defn}

It was shown by exploiting the convex integration method \cite{dls09}, \cite{dls10} that for some initial states $(\varrho_\pm,v_\pm)\in\mathbb{R}^+ \times\mathbb{R}^2$ there exist infinitely many admissible weak solutions, see \cite{chio14}, \cite{chio15} and \cite{mk17}. These infinitely many solutions that are produced by the convex integration method are usually called \emph{wild solutions}. In this note we want to show that there are initial states $(\varrho_\pm,v_\pm)\in\mathbb{R}^+ \times\mathbb{R}^2$ for which there are infinitely many weak solutions that fulfill the energy equation \eqref{eq:eneq} rather than the energy inequality \eqref{eq:enineq}. In other words we will prove that there are not only infinitely many admissible weak solutions but also infinitely many energy conserving weak solutions.

The following theorem is our main result:

\begin{thm}
	Let $p(\varrho)=\varrho^2$. There exist initial states $(\varrho_\pm,v_\pm)\in\mathbb{R}^+ \times\mathbb{R}^2$ such that there are infinitely many energy conserving weak solutions to \eqref{eq:euler}, \eqref{eq:riemann}.
\label{thm:1}
\end{thm}

This theorem will be proved by methods first presented in \cite{chio14} and \cite{chio15}. 

As a consequence of theorem \ref{thm:1} we can show the following:

\begin{cor}
	Let $p(\varrho)=\varrho^2$. There exist Lipschitz continuous initial data $(\varrho_0,v_0)$ such that there are infinitely many energy conserving weak solutions to \eqref{eq:euler}.
\label{cor:2}
\end{cor}

\section{Proof of theorem \ref{thm:1}}

\subsection{Definitions}
We proceed as in \cite{chio14} and therefore recall the definition of a fan partition. 

\begin{defn} \emph{(fan partition, see \cite[Definition 4]{chio14})}
	Let \mbox{$\mu_0<\mu_1$} real numbers. A fan partition of $(0,\infty)\times\mathbb{R}^2$ consists of three open sets $\Omega_-,\Omega_1,\Omega_+$ of the form
	\begin{align*}
	\Omega_-&=\{(t,x):t>0\text{ and }x_2<\mu_0\,t\}, \\
	\Omega_1&=\{(t,x):t>0\text{ and }\mu_{0}\,t<x_2<\mu_1\,t\}, \\
	\Omega_+&=\{(t,x):t>0\text{ and }x_2>\mu_1\,t\}.
	\end{align*}
	\label{defn:fanpart}
\end{defn}

Furthermore we define 
\begin{align*}
	\S&:=\{M\in\mathbb{R}^{2\times 2}\,|\,M\text{ symmetric}\}\text{ and } \\ \Sz&:=\{M\in\S\,|\,\tr(M)=0\}.
\end{align*}

We also recall the definition of a fan subsolution, where we slightly adjust \cite[Defnitions 5 and 6]{chio14}
for our needs. More precisely we have an equality in \eqref{eq:fansubs4} (see below) in contrast to the inequality in \cite[Definition 6]{chio14}.

\begin{defn} \emph{(energy conserving fan subsolution, cf. \cite[Definitions 5 and 6]{chio14})}
	An energy conserving fan subsolution to the Euler system \eqref{eq:euler} with initial condition \eqref{eq:riemann} is a triple $(\overline{\varrho},\overline{v},\overline{u}):(0,\infty)\times\mathbb{R}^2\rightarrow(\mathbb{R}^+\times\mathbb{R}^2\times\Sz)$ of piecewise constant functions, which satisfies the following properties:
	\begin{enumerate}
		\item There exists a fan partition of $(0,\infty)\times\mathbb{R}^2$ and constants $\varrho_1\in\mathbb{R}^+$, $v_1\in\mathbb{R}^2$ and $u_1\in\Sz$, such that
		\begin{align*}
		(\overline{\varrho},\overline{v},\overline{u})&=\sum\limits_{i\in\{-,+\}} \bigg(\varrho_i\,,\,v_i\,,\,v_i\otimes v_i - \frac{|v_i|^2}{2}\,\id\bigg)\,\mathbf{1}_{\Omega_i} + (\varrho_1\,,\,v_1\,,\,u_1)\,\mathbf{1}_{\Omega_1},
		\end{align*}
		where $(\varrho_{\pm},v_{\pm})$ are the given initial states. 
		\item There is a constant $C_1\in\mathbb{R}^+$ such that
		\begin{equation*}
		v_1\otimes v_1-u_1 < \frac{C_1}{2}\,\id
		\end{equation*}
		in the sense of definiteness.
		\item For all test functions $(\psi,\phi)\in C_c^\infty([0,\infty)\times\mathbb{R}^2,\mathbb{R}\times\mathbb{R}^2)$ the following identities hold:
		\begin{align*}
		\int_0^\infty\int_{\mathbb{R}^2}\big(\overline{\varrho}\,\partial_t\psi + \overline{\varrho}\,\overline{v}\cdot\nabla_x\psi\big)\dd x\,\dd t + \int_{\mathbb{R}^2} \varrho_0(x)\,\psi(0,x)\,\dd x\ &=\ 0, \\
		\int_0^\infty\int_{\mathbb{R}^2}\Bigg[\overline{\varrho}\,\overline{v}\cdot\partial_t\phi + \overline{\varrho}\,\Big(\big(\overline{v}\otimes \overline{v}\big)\,\mathbf{1}_{\Omega_- \cup \Omega_+} + u_1\,\mathbf{1}_{\Omega_1}\Big):D_x\phi\qquad\qquad\qquad&  \\
		+ \bigg(p(\overline{\varrho}) +\frac{1}{2}\,\varrho_1\,C_1\,\mathbf{1}_{\Omega_1}\bigg)\,\text{div}_x\phi\Bigg]\dd x\,\dd t  + \int_{\mathbb{R}^2} \varrho_0(x)\,v_0(x)\cdot\phi(0,x)\,\dd x\  &=\ 0. 
		\end{align*}
		\item For every non-negative test function $\varphi\in C_c^\infty([0,\infty)\times\mathbb{R}^2,\mathbb{R}_0^+)$ the equation
		\begin{align}
		&\int_0^\infty\int_{\mathbb{R}^2}\Bigg[\bigg(P(\overline{\varrho}) + \frac{1}{2}\,\overline{\varrho}\,\Big(|\overline{v}|^2\,\mathbf{1}_{\Omega_- \cup \Omega_+} + C_1\,\mathbf{1}_{\Omega_1}\Big)\bigg)\,\partial_t\varphi \notag\\
		&\qquad+ \bigg(P(\overline{\varrho})+p(\overline{\varrho}) + \frac{1}{2}\,\overline{\varrho}\,\Big(|\overline{v}|^2\,\mathbf{1}_{\Omega_- \cup \Omega_+} + C_1\,\mathbf{1}_{\Omega_1}\Big)\bigg)\,\overline{v}\cdot\nabla_x\varphi\Bigg]\dd x\,\dd t \notag\\
		&\quad+ \int_{\mathbb{R}^2} \bigg(P(\varrho_0(x)) + \varrho_0(x)\,\frac{|v_0(x)|^2}{2}\bigg)\,\varphi(0,x)\,\dd x \quad= \quad 0 
		\label{eq:fansubs4}
		\end{align}
		is fulfilled.
	\end{enumerate}
	\label{defn:fansubs}
\end{defn}

\subsection{Sufficient condition for non-uniqueness}

We obtain a slightly adjusted version of \cite[Proposition 3.1]{chio14}.

\begin{thm} \emph{(cf. \cite[Proposition 3.1]{chio14})} 
	Let $(\varrho_\pm,v_\pm)$ be such that there exists an energy conserving fan subsolution $(\overline{\varrho},\overline{v},\overline{u})$ to \eqref{eq:euler}, \eqref{eq:riemann}. Then there are infinitely many energy conserving weak solutions $(\varrho,v)$ to \eqref{eq:euler}, \eqref{eq:riemann} with the following properties:
	\begin{itemize}
		\item $\varrho=\overline{\varrho}$,
		\item $v(t,x)=\overline{v}(t,x)$ for almost all $(t,x)\in \Omega_- \cup \Omega_+$, 
		\item $|v(t,x)|^2=C_1$ for almost all $(t,x)\in \Omega_1$.
	\end{itemize}
	\label{thm:condition}
\end{thm}

\begin{proof}
	For the proof we refer to the proofs of \cite[Proposition 3.1]{chio14} and \cite[Proposition 3.6]{chio15}. Note that the energy conservation \eqref{eq:fansubs4} of the fan subsolution ensures that the energy equation \eqref{eq:eneq} holds, i.e. that the weak solutions are energy conserving.
\end{proof}

\subsection{Algebraic equations and inequalities}
As in \cite[Proposition 4.1]{chio14}, definition \ref{defn:fansubs} can be translated into a system of algebraic equations and inequalities. Note again that we get here energy equations instead of energy inequalities.

\begin{prop} \emph{(cf. \cite[Proposition 4.1]{chio14})} 
	Let $(\varrho_\pm,v_\pm)$ be given. The constants $\mu_0,\mu_1\in\mathbb{R}$, $\varrho_1\in\mathbb{R}^+$, 
	\begin{align*}
	v_1&=\left(
	\begin{array}[c]{l}
	v_{1\,1}\\
	v_{1\,2}
	\end{array}\right) \in\mathbb{R}^2, &u_1&=\left(
	\begin{array}[c]{rr}
	u_{1\,11} & u_{1\,12} \\
	u_{1\,12} & -u_{1\,11}
	\end{array} \right)\in\Sz
	\end{align*}
	and $C_1\in\mathbb{R}^+$ define an energy conserving fan subsolution to \eqref{eq:euler}, \eqref{eq:riemann} if and only if they fulfill the following algebraic equations and inequalities: 
	\begin{itemize}
		\item Order of the speeds:
		\begin{equation}
		\mu_0<\mu_1
		\label{eq:order}
		\end{equation}
		\item Rankine Hugoniot conditions on the left interface:
		\begin{align}
		\mu_0\,(\varrho_- - \varrho_1) &= \varrho_-\,v_{-\,2} - \varrho_1\,v_{1\,2} \label{eq:rhl1}\\
		\mu_0\,(\varrho_-\,v_{-\,1} - \varrho_1\,v_{1\,1}) &= \varrho_-\,v_{-\,1}\,v_{-\,2} - \varrho_1\,u_{1\,12} \label{eq:rhl2}\\
		\mu_0\,(\varrho_-\,v_{-\,2} - \varrho_1\,v_{1\,2}) &= \varrho_-\,v_{-\,2}^2 + \varrho_1\,u_{1\,11} + p(\varrho_-) - p(\varrho_1) - \varrho_1\,\frac{C_1}{2} \label{eq:rhl3}
		\end{align}
		\item Rankine Hugoniot conditions on the right interface:
		\begin{align}
		\mu_1\,(\varrho_1 - \varrho_+) &= \varrho_1\,v_{1\,2} - \varrho_+\,v_{+\,2} \label{eq:rhr1}\\
		\mu_1\,(\varrho_1\,v_{1\,1} - \varrho_+\,v_{+\,1}) &= \varrho_1\,u_{1\,12} - \varrho_+\,v_{+\,1}\,v_{+\,2} \label{eq:rhr2}\\
		\mu_1\,(\varrho_1\,v_{1\,2} - \varrho_+\,v_{+\,2}) &= - \varrho_1\,u_{1\,11} - \varrho_+\,v_{+\,2}^2 + p(\varrho_1) - p(\varrho_+) + \varrho_1\,\frac{C_1}{2} \label{eq:rhr3}
		\end{align}
		\item Subsolution condition:
		\begin{align}
		v_{1\,1}^2 + v_{1\,2}^2 &< C_1 \label{eq:sc1}\\
		\bigg(\frac{C_1}{2} - v_{1\,1}^2 + u_{1\,11}\bigg) \bigg(\frac{C_1}{2} - v_{1\,2}^2 - u_{1\,11}\bigg) - (u_{1\,12}-v_{1\,1}\,v_{1\,2})^2 &> 0 \label{eq:sc2}
		\end{align}
		\item Energy equation on the left interface:
		\begin{align}
		&\mu_0\,\bigg(P(\varrho_-)+\varrho_-\,\frac{|v_-|^2}{2}-P(\varrho_1)-\varrho_1\,\frac{C_1}{2}\bigg) \notag\\
		&=\big(P(\varrho_-)+p(\varrho_-)\big)\,v_{-\,2}-\big(P(\varrho_1)+p(\varrho_1)\big)\,v_{1\,2}+\varrho_-\,v_{-\,2}\,\frac{|v_-|^2}{2}-\varrho_1\,v_{1\,2}\,\frac{C_1}{2}
		\label{eq:enl}
		\end{align}
		\item Energy equation on the right interface:
		\begin{align}
		&\mu_1\,\bigg(P(\varrho_1)+\varrho_1\,\frac{C_1}{2}-P(\varrho_+)-\varrho_+\,\frac{|v_+|^2}{2}\bigg) \notag\\
		&=\big(P(\varrho_1)+p(\varrho_1)\big)\,v_{1\,2}-\big(P(\varrho_+)+p(\varrho_+)\big)\,v_{+\,2}+\varrho_1\,v_{1\,2}\,\frac{C_1}{2}-\varrho_+\,v_{+\,2}\,\frac{|v_+|^2}{2}
		\label{eq:enr}
		\end{align}
	\end{itemize}
	\label{prop:alggl}
\end{prop}

\begin{proof}
	The above proposition can be proved analogously to \cite[Proposition 5.1]{chio15}.
\end{proof}

\subsection{Proof}

Finally we prove theorem \ref{thm:1}.

\begin{proof}
	Let the initial data be given by 
	\begin{align}
		\varrho_- &= 1 & \varrho_+ &= 4 \notag\\
		v_- &= \left(\begin{array}{c}
			0 \\
			2\sqrt{2}
		\end{array}\right) & v_+ &= \left(\begin{array}{c}
			0 \\
			0
		\end{array}\right).
	\label{eq:specialdata}
	\end{align}
	
	We set 
	\begin{align} 
		\mu_0 &= -\frac{7}{2\sqrt{2}} &
		\mu_1 &= 0 \notag\\
		\varrho_1 &= \frac{15}{7} &
		v_1 &= \left(\begin{array}{c}
			0 \\
			0
		\end{array} \right) \notag\\
		u_1 &= \left(\begin{array}{cc}
			-\frac{29}{15} & 0\\
			0 & \frac{29}{15}
		\end{array}\right) &
		C_1 &= \frac{712}{105}.
	\label{eq:values}
	\end{align}
	
	Simple computations show that the equations and inequalities of proposition \ref{prop:alggl} hold. Hence there exists an energy conserving fan subsolution and according to theorem \ref{thm:condition} infinitely many energy conserving weak solutions to \eqref{eq:euler} with initial data \eqref{eq:riemann}, \eqref{eq:specialdata}.
\end{proof}

The reader might wonder where the values in \eqref{eq:values} come from. Hence we want to expose the way we reached to them. 

Let us begin by looking for weak solutions to \eqref{eq:euler} with initial data \eqref{eq:riemann}, \eqref{eq:specialdata} that are \emph{admissible} instead of \emph{energy conserving}. This means that the energy inequality \eqref{eq:enineq} holds rather than the energy equation \eqref{eq:eneq}. To this end, as shown in \cite{chio15}, \cite{chio14} and \cite{mk17}, one has to look for \emph{admissible} fan subsolutions. Admissible fan subsolutions are similar to energy conserving fan subsolutions as defined in definition \ref{defn:fansubs}, with the only difference that the energy equation \eqref{eq:fansubs4} turns into an energy inequality. One ends up with algebraic equations and inequalities, see e.g. \cite[Proposition 5.1]{chio15}, like those presented in proposition \ref{prop:alggl} with the difference that again the energy equations \eqref{eq:enl}, \eqref{eq:enr} are replaced by energy inequalities. Since there are six equations for eight unknowns, the idea in \cite{chio14} and also in \cite{mk17} was to choose two unknowns as parameters and express the other six unknowns as functions of the two parameters. For convenience one replaces the unknowns $u_{1\,11}$ and $C_1$ by 
\begin{align*}
	\delta_1 &= \frac{C_1}{2}-v_{1\,2}^2 - u_{1\,11}, & \delta_2 &= \frac{C_1}{2}-v_{1\,1}^2 + u_{1\,11}
\end{align*}
and chooses $\varrho_1,\delta_2$ as parameters. For the special initial data \eqref{eq:specialdata} one obtains the following proposition, shown by the authors in \cite{mk17}. 

\begin{prop} \emph{(see \cite[Theorem 5.2]{mk17})} 
	There exists an admissible fan subsolution to \eqref{eq:euler} with initial data \eqref{eq:riemann}, \eqref{eq:specialdata} if there exist constants $\varrho_1,\delta_2\in\mathbb{R}^+$ that fulfill
	\begin{align} 
	\varrho_- &< \varrho_1 < \varrho_+, \label{eq:SRcond1} \\[0.3cm]
	\delta_1(\varrho_1) &>0, \label{eq:SRcond2} 
	\end{align} 
	\vspace{-0.7cm}
	\begin{align}
	\begin{split}
	&(v_{1\,2}(\varrho_1)-v_{-\,2})\,\bigg(p(\varrho_-)+p(\varrho_1)-2\,\frac{\varrho_1\,P(\varrho_-)-\varrho_-\,P(\varrho_1)}{\varrho_- -\varrho_1}\bigg) \\
	&\qquad\leq\delta_1(\varrho_1)\,\varrho_1\,(v_{1\,2}(\varrho_1)+v_{-\,2})-(\delta_1(\varrho_1)+\delta_2)\,\frac{\varrho_-\,\varrho_1\,(v_{1\,2}(\varrho_1)-v_{-\,2})}{\varrho_- -\varrho_1},
	\end{split} \label{eq:SRcond3} \\[0.5cm] 
	\begin{split}
	&(v_{+\,2}-v_{1\,2}(\varrho_1))\,\bigg(p(\varrho_1)+p(\varrho_+)-2\,\frac{\varrho_+\,P(\varrho_1)-\varrho_1\,P(\varrho_+)}{\varrho_1 -\varrho_+}\bigg) \\
	&\qquad\leq-\delta_1(\varrho_1)\,\varrho_1\,(v_{+\,2}+v_{1\,2}(\varrho_1))+(\delta_1(\varrho_1)+\delta_2)\,\frac{\varrho_1\,\varrho_+\,(v_{+\,2}-v_{1\,2}(\varrho_1))}{\varrho_1 -\varrho_+},
	\end{split} \label{eq:SRcond4}
	\end{align}
	where we define the functions
	\begin{align}
	&v_{1\,2}(\varrho_1) := \frac{1}{\varrho_1\,(\varrho_- - \varrho_+)}\,\Bigg(-\varrho_-\,v_{-\,2}\,(\varrho_+ - \varrho_1) - \varrho_+\,v_{+\,2}\,(\varrho_1 - \varrho_-) \notag\\
	&\ \ + \sqrt{\Big[(\varrho_- - \varrho_+)\,\big(p(\varrho_-) - p(\varrho_+)\big) - \varrho_+\,\varrho_-\,(v_{-\,2} - v_{+\,2})^2\Big]\,(\varrho_1 - \varrho_-)\,(\varrho_+ - \varrho_1)}\Bigg)  \label{eq:SRdefnv12}
	\end{align}
	and
	\begin{equation}
	\begin{split}
	&\delta_1(\varrho_1) := -\frac{p(\varrho_1) - p(\varrho_-)}{\varrho_1} + \frac{\varrho_-\,(\varrho_1 - \varrho_-)}{\varrho_1^2\,(\varrho_- - \varrho_+)^2}\,\Bigg(\varrho_+\,(v_{-\,2} - v_{+\,2}) \\
	&\ \ + \sqrt{\Big[(\varrho_- - \varrho_+)\,\big(p(\varrho_-) - p(\varrho_+)\big) - \varrho_+\,\varrho_-\,(v_{-\,2} - v_{+\,2})^2\Big]\,\frac{\varrho_+ - \varrho_1}{\varrho_1 - \varrho_-}}\Bigg)^2.  
	\end{split}
	\label{eq:SRdefndelta1}
	\end{equation} 
	
	Note that these functions are well-defined for $\varrho_- < \varrho_1 < \varrho_+$ and for the initial states \eqref{eq:specialdata}.
\label{prop:exadmfansubs}
\end{prop}

Here the conditions $\delta_2>0$ and \eqref{eq:SRcond2} ensure the subsolution conditions, \eqref{eq:SRcond1} guarantees the correct order of the speeds and the inequalities \eqref{eq:SRcond3} and \eqref{eq:SRcond4} correspond to the energy inequalities for the left and the right interface, respectively (see proof of \cite[Theorem 5.2]{mk17}).

Next we are going to plot the regions in the $\varrho_1$-$\delta_2$-plane where the inequalities \eqref{eq:SRcond1} - \eqref{eq:SRcond4} are fulfilled, see figure \ref{fig:regions}. Proposition \ref{prop:exadmfansubs} claims that each point $(\varrho,\delta_2)$ lying in the region shown in figure \ref{fig:regions} (d) corresponds to an admissible fan subsolution. 

\begin{figure}[hbt]
	\centering
	\subfigure[Condition \eqref{eq:SRcond3}]{\includegraphics[width=0.35\textwidth]{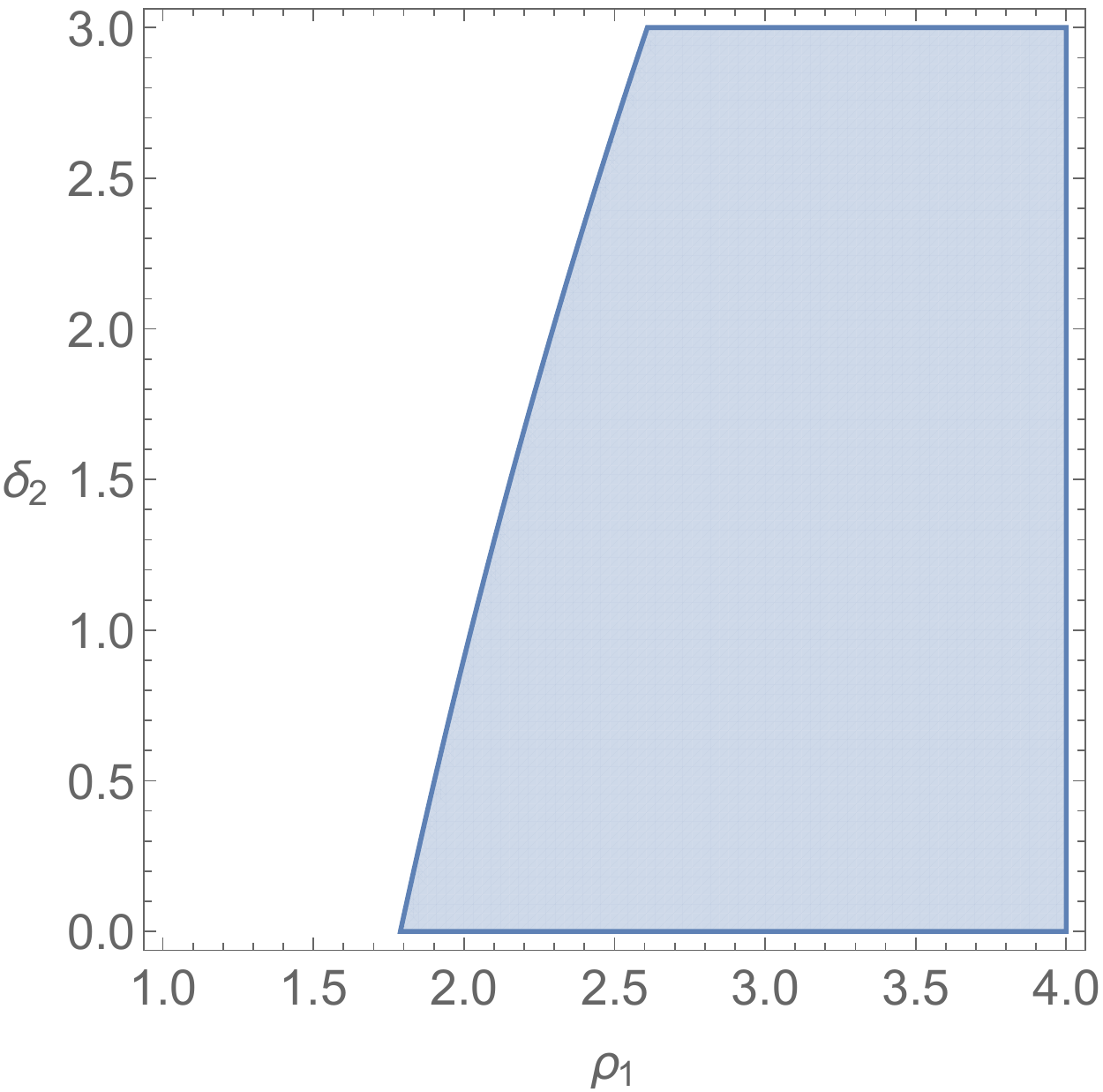}}\hspace{30pt}
	\subfigure[Condition \eqref{eq:SRcond4}]{\includegraphics[width=0.35\textwidth]{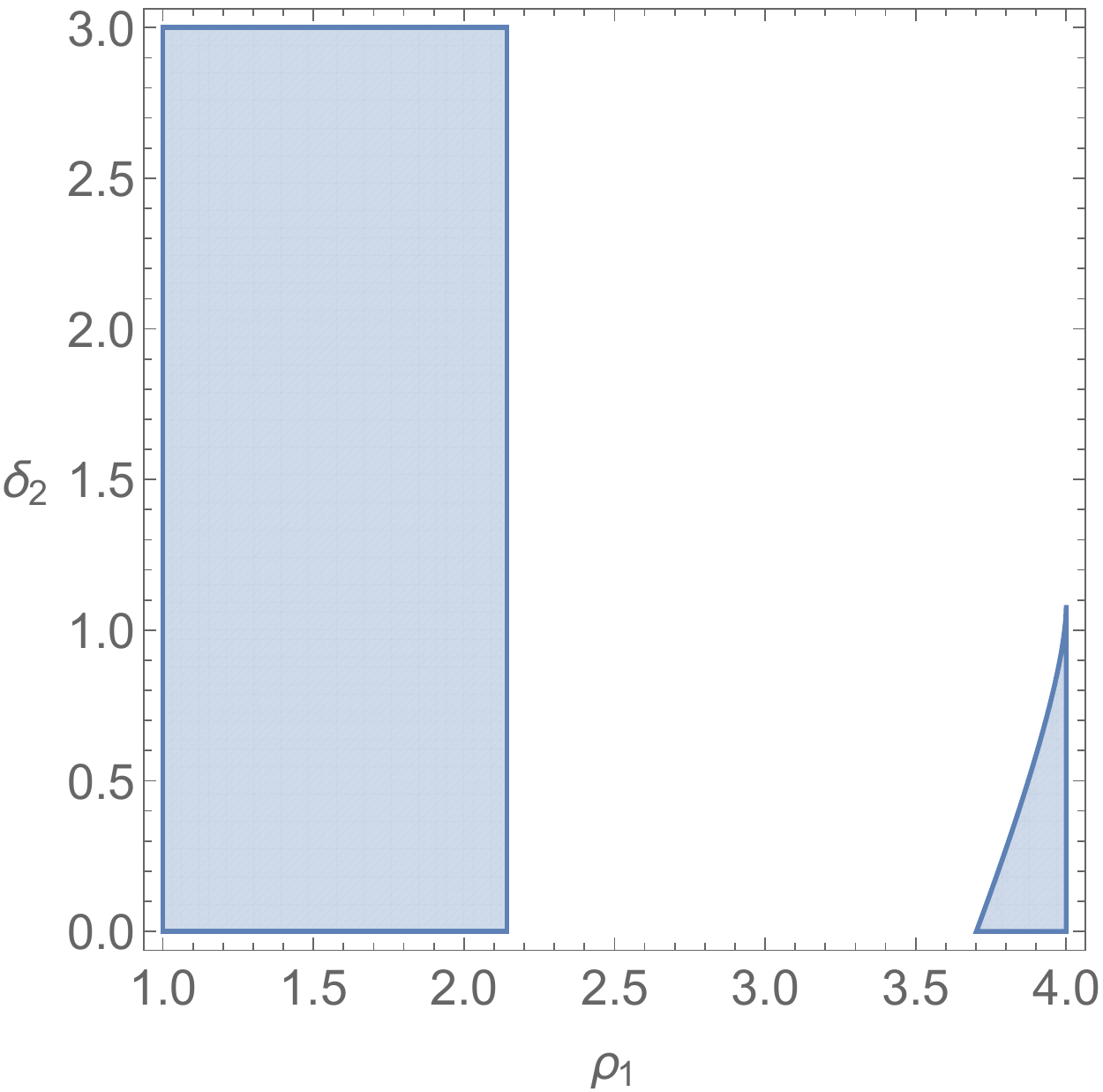}} \\
	\subfigure[Condition \eqref{eq:SRcond2}]{\includegraphics[width=0.35\textwidth]{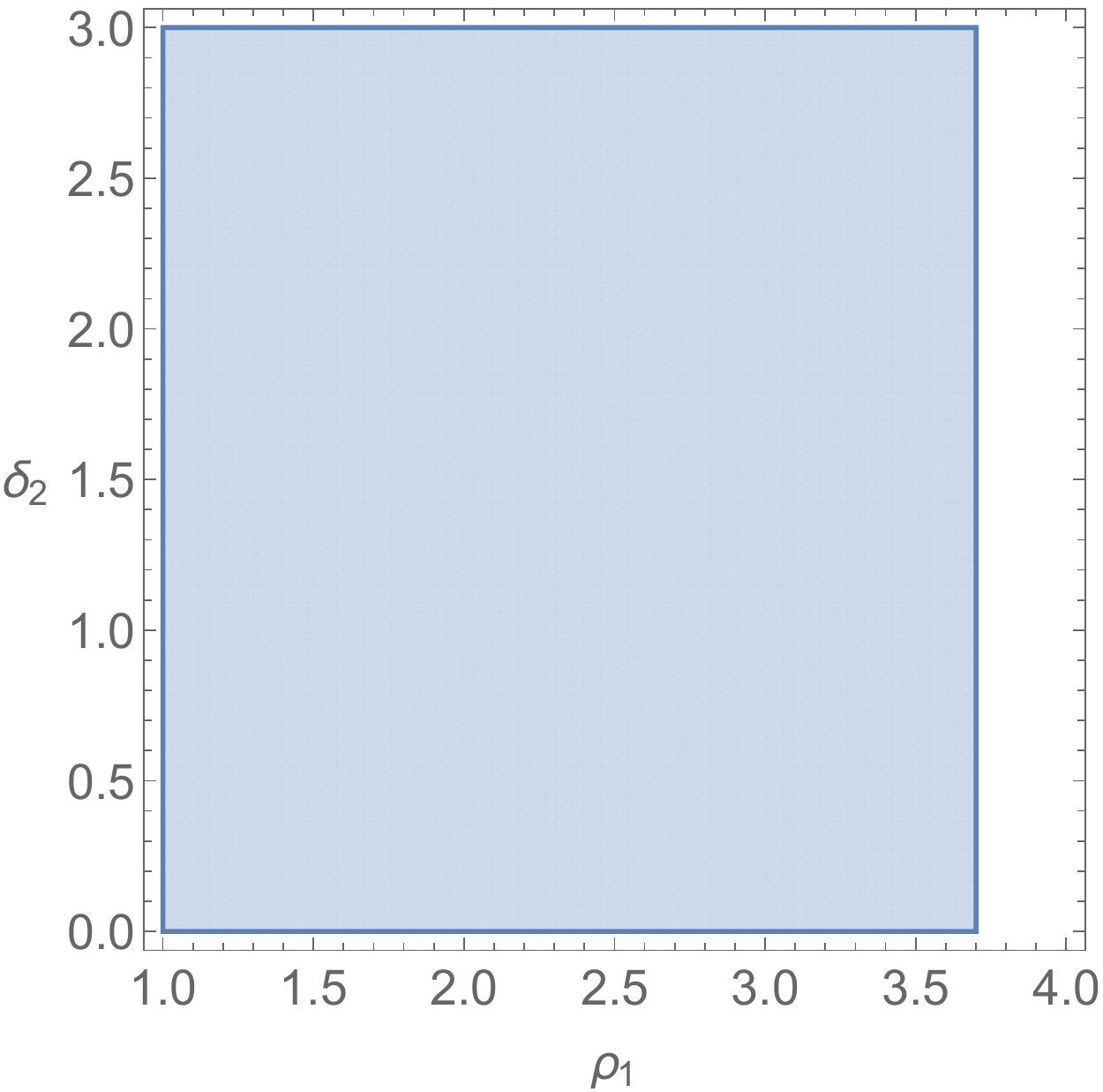}}\hspace{30pt} 
	\subfigure[Condition \eqref{eq:SRcond1} - \eqref{eq:SRcond4}]{\includegraphics[width=0.35\textwidth]{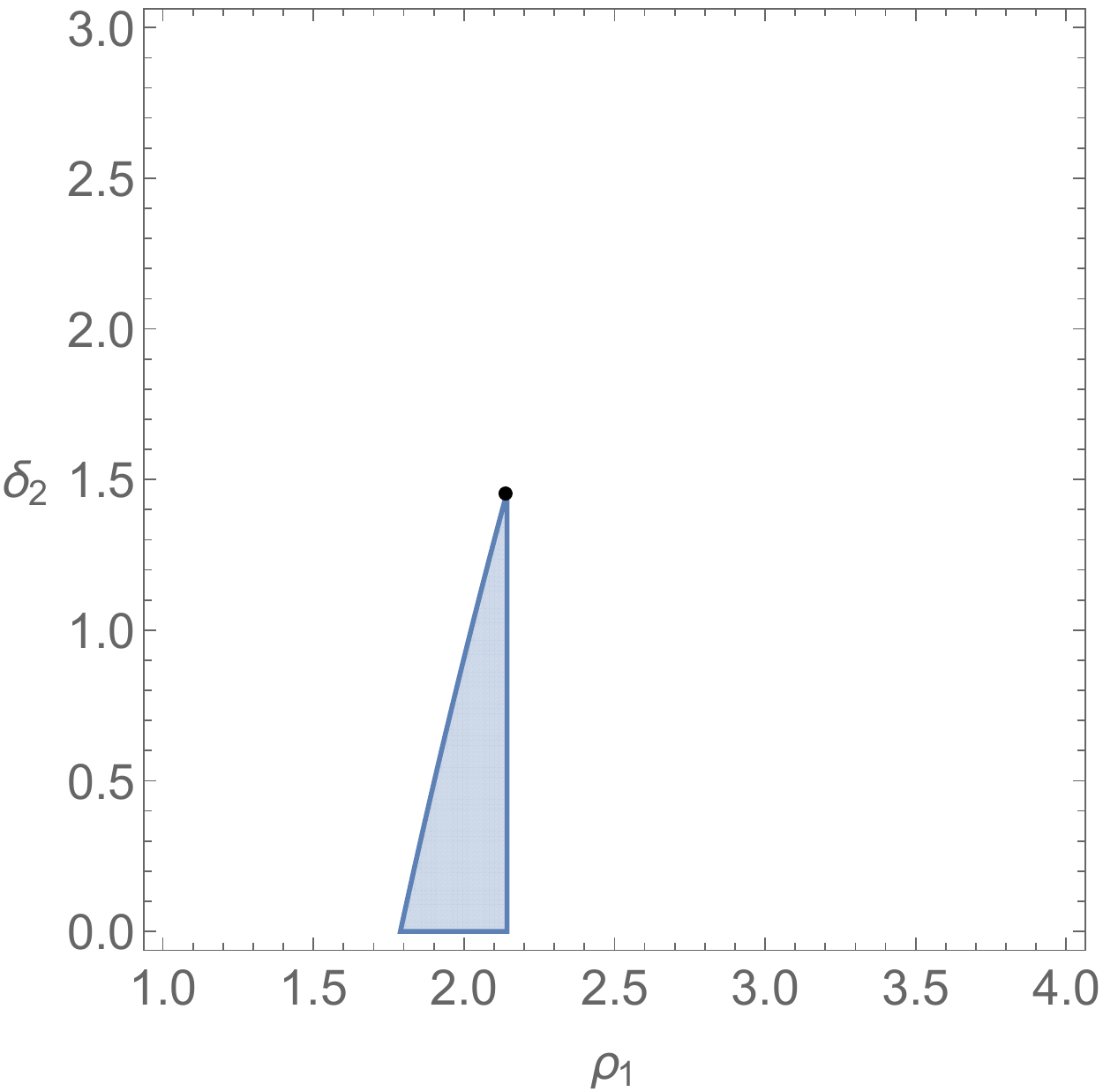}}
	\caption{Regions in the $\varrho_1$-$\delta_2$-plane where the conditions of proposition \ref{prop:exadmfansubs} hold}
\label{fig:regions}
\end{figure}

Let us now return to this paper's aim, namely to find energy conserving solutions. The procedure presented above will lead to a similar claim as stated in proposition \ref{prop:exadmfansubs} with the difference that the inequalities \eqref{eq:SRcond3}, \eqref{eq:SRcond4} are replaced by equations. Therefore, in order to find an energy conserving fan subsolution, we have to find a point $(\varrho_1,\delta_2)$ in the region shown in figure \ref{fig:regions} (c) which in addition lies on the boundary of the regions shown in figures \ref{fig:regions} (a) and (b). Hence it suffices to find the apex in figure \ref{fig:regions} (d). The values \eqref{eq:values} correspond exactly to this apex.

\section{Proof of corollary \ref{cor:2}}

\begin{proof}
	In order to prove corollary \ref{cor:2} we proceed as in the proof of \cite[Corollary 1.2]{chio15}. We solve the system \eqref{eq:euler} with initial data \eqref{eq:riemann}, \eqref{eq:specialdata} backwards in time. It is an easy observation that if $\big(\varrho(t,x),v(t,x)\big)$ is a solution to \eqref{eq:euler} then 
	\begin{equation*}
		\big(\widetilde{\varrho}(t,x),\widetilde{v}(t,x)\big) = \big(\varrho(-t,-x),v(-t,-x)\big)
	\end{equation*}
	solves \eqref{eq:euler}, too. Hence in order to solve \eqref{eq:euler} with initial data \eqref{eq:riemann}, \eqref{eq:specialdata} backwards in time, it suffices to solve \eqref{eq:euler} forward in time, where now the initial states are switched 
	\begin{equation}
	\begin{split}
		(\varrho,v)(0,x) &= (\widetilde{\varrho}_0,\widetilde{v}_0)(x):= (\varrho_0,v_0)(-x)=\left\{
		\begin{array}[c]{ll}
		(\varrho_+,v_+) & \text{ if }x_2<0 \\
		(\varrho_-,v_-) & \text{ if }x_2>0
		\end{array}
		\right. ,
	\end{split}
	\label{eq:switchedriemann}
	\end{equation}
	and $\varrho_\pm,v_\pm$ given in \eqref{eq:specialdata}. 
	
	By well-known methods one can show that there is a self-similar solution $(\widetilde{\varrho},\widetilde{v})$ to \eqref{eq:euler}, \eqref{eq:switchedriemann} which consists of one rarefaction, see \cite[Proposition 1.3]{mk17} and the references therein. Note that rarefaction solutions are Lipschitz in the spatial variable $x$ and they conserve the energy in the sense of \eqref{eq:eneq}. Hence we can simply set the initial data to 
	\begin{equation*}
		(\varrho_0,v_0)(x) := (\widetilde{\varrho},\widetilde{v})(t=1,x),
 	\end{equation*}
 	which is Lipschitz and leads to Riemann data \eqref{eq:riemann}, \eqref{eq:specialdata} for $t=1$. Consequently according to theorem \ref{thm:1} there exist infinitely many energy conserving weak solutions to \eqref{eq:euler} with Lipschitz initial data as above. All these solutions coincide for $t\in[0,1]$.
\end{proof}

\section{Conclusion}
In this note we showed that even if we insist that solutions conserve energy for the two-dimensional isentropic Euler equations one can find Lipschitz initial data which gives rise to infinitely many weak solutions. The proof makes use of the techniques shown in \cite{mk17}. Thus if one wants to find a criterion to rule out these solutions one has to look elsewhere.

\section*{Acknowledgement}
Both authors were partially funded by the European Research Council under the European Union’s Seventh Framework Programme (FP7/2007-2013)/ ERC Grant Agreement 320078. We thank Eduard Feireisl for hosting us.


\newpage

\begin{thebibliography}{1}
	\providecommand{\url}[1]{{#1}}
	\providecommand{\urlprefix}{URL }
	\expandafter\ifx\csname urlstyle\endcsname\relax
	\providecommand{\doi}[1]{DOI~\discretionary{}{}{}#1}\else
	\providecommand{\doi}{DOI~\discretionary{}{}{}\begingroup
		\urlstyle{rm}\Url}\fi
	
	\bibitem{chen16}
	Chen, G.-Q., Chen, J., Feldman, M.: Transonic flows with shocks past curved wedges for the full Euler equations. 
	\newblock Discrete Contin. Dyn. Syst. \textbf{36} (8) 4179--4211 (2016)
	
	\bibitem{chio15}
	Chiodaroli, E., De Lellis, C., Kreml, O.: Global ill-posedness of the isentropic
	system of gas dynamics.
	\newblock Comm. Pure Appl. Math. \textbf{68} (7), 1157--1190 (2015)
	
	\bibitem{chio14}
	Chiodaroli, E., Kreml, O.: On the energy dissipation rate of solutions to the
	compressible isentropic euler system.
	\newblock Arch. Ration. Mech. Anal. \textbf{214} (3), 1019--1049 (2014)
	
	\bibitem{diperna83}
	DiPerna, R. J.: Convergence of approximate solutions to conservation laws. 
	\newblock Arch. Ration. Mech. Anal. \textbf{82} (1), 27--70 (1983)
	
	\bibitem{dls09}
	De Lellis, C., Sz{\'{e}}kelyhidi Jr., L.: The euler equations as a differential
	inclusion.
	\newblock Ann. of Math. (2) \textbf{170} (3), 1417--1436 (2009)
	
	\bibitem{dls10}
	De Lellis, C., Sz{\'{e}}kelyhidi Jr., L.: On admissibility criteria for weak
	solutions of the euler equations.
	\newblock Arch. Ration. Mech. Anal. \textbf{195} (1), 225--260 (2010)
		
	\bibitem{glimm65}
	Glimm, J.: Solutions in the large for nonlinear hyperbolic systems of equations.
	\newblock Comm. Pure Appl. Math. \textbf{18}  (4) 697--715 (1965)

	\bibitem{mk17}
	Markfelder, S., Klingenberg, C.: The Riemann problem for the multidimensional isentropic system of gas dynamics is ill-posed if it contains a shock.
	\newblock accepted in Arch. Ration. Mech. Anal., see also arXiv: 1708.01063 (2017)
	
\end{thebibliography}
\end{document}